\documentclass[reqno]{amsart}

\usepackage{color}

\newtheorem{theorem}{Theorem}[section]
\newtheorem{lemma}[theorem]{Lemma}
\newtheorem{corollary}[theorem]{Corollary}
\newtheorem{proposition}[theorem]{Proposition}

\theoremstyle{definition}
\newtheorem{definition}[theorem]{Definition}
\newtheorem{example}[theorem]{Example}

\theoremstyle{remark}
\newtheorem{remark}[theorem]{Remark}

\numberwithin{equation}{section}

\begin{document}

\title{Regular subspaces of skew product diffusions}

\author{Liping Li}
\address{School of Mathematical Sciences, Fudan University, 220 Handan Road, Shanghai, China.  200433}
\email{lipingli10@fudan.edu.cn}
\thanks{Research supported in part by NSFC Grant 11271240.}

\author{Jiangang Ying}
\address{School of Mathematical Sciences, Fudan University, 220 Handan Road, Shanghai, China.  200433}
\email{jgying@fudan.edu.cn}





\begin{abstract}
Roughly speaking, the regular subspace of a Dirichlet form is also a regular Dirichlet form on the same state space. It inherits the same form of original Dirichlet form but possesses a smaller domain.
What we are concerned in this paper are the regular subspaces of associated Dirichlet forms of skew product diffusions.
A skew product diffusion $X$ is a symmetric Markov process on the product state space $E_1\times E_2$ and expressed as
\[
	X_t=(X^1_t,X^2_{A_t}),\quad t\geq 0,
\]
where $X^i$ is a symmetric diffusion on $E_i$ for $i=1,2$, and $A$ is a positive continuous additive functional of $X^1$. 
 One of our main results indicates that any skew product type regular subspace of $X$, say
 \[
 	Y_t=(Y^1_t,Y^2_{\tilde{A}_t}),\quad t\geq 0,
 \] 
can be characterized as follows: the associated smooth measure of $\tilde{A}$ is equal to that of $A$, and $Y^i$ corresponds to a regular subspace of $X^i$ for $i=1,2$. Furthermore, we shall make some discussions on rotationally invariant diffusions on $\mathbf{R}^d\setminus \{0\}$, which are special skew product diffusions on $(0,\infty)\times S^{d-1}$. Our main purpose is to extend a regular subspace of rotationally invariant diffusion on $\mathbf{R}^d\setminus \{0\}$ to a new regular Dirichlet form on $\mathbf{R}^d$. More precisely, fix a regular Dirichlet form $(\mathcal{E,F})$ on the state space $\mathbf{R}^d$. Its part Dirichlet form on $\mathbf{R}^d\setminus \{0\}$ is denoted by  $(\mathcal{E}^0,\mathcal{F}^0)$. Let $(\tilde{\mathcal{E}}^0,\tilde{\mathcal{F}}^0)$ be a regular subspace of $(\mathcal{E}^0,\mathcal{F}^0)$. We want to find a regular subspace $(\tilde{\mathcal{E}},\tilde{\mathcal{F}})$ of $(\mathcal{E,F})$ such that the part Dirichlet form of $(\tilde{\mathcal{E}},\tilde{\mathcal{F}})$ on $\mathbf{R}^d\setminus \{0\}$ is exactly $(\tilde{\mathcal{E}}^0,\tilde{\mathcal{F}}^0)$. If $(\tilde{\mathcal{E}},\tilde{\mathcal{F}})$ exists, we call it a regular extension of $(\tilde{\mathcal{E}}^0,\tilde{\mathcal{F}}^0)$. We shall prove that under a mild assumption, any rotationally invariant type regular subspace of $(\mathcal{E}^0,\mathcal{F}^0)$ has a unique regular extension. 
\end{abstract}

\maketitle



\section{Introduction}
\label{intro}

The regular subspaces of a Dirichlet form were first raised in  \cite{FMG} and then concerned in \cite{XPJ}, \cite{FL}, \cite{LY2} and \cite{LY} by the second author and his co-authors. To introduce this conception, let $E$ be a locally compact {separable metric space} and $m$ a  Radon measure fully supported on $E$. A non-negative definite symmetric bilinear form $\mathcal{E}$ with domain $\mathcal{F}$ densely defining on $L^2(E,m)$ is called a \emph{Dirichlet form}, which is denoted by $(\mathcal{E},\mathcal{F})$, if it is closed and Markovian. Further let $(T_t)_{t\geq 0}, (G_\alpha)_{\alpha\geq 0}$ be its $L^2$-semigroup and $L^2$-resolvent. Define
 \[
	 \mathcal{E}_\alpha(u,v):= \mathcal{E}(u,v)+\alpha\cdot (u,v)_m,\quad {u,v\in \mathcal{F}}, \alpha\geq 0.
 \]
We denote  the space of all real continuous functions on $E$ by $C(E)$ and its subspace of continuous functions with compact support (resp. bounded continuous functions, continuously differentiable functions with compact supports, continuous functions which converge to zero at infinity) by $C_\mathrm{c}(E)$ (resp. $C_\mathrm{b}(E),C^1_\mathrm{c}(E), C_0(E)$). A Dirichlet form $(\mathcal{E},\mathcal{F})$ is called \emph{regular} if $\mathcal{F}\cap C_\mathrm{c}(E)$ is dense in $\mathcal{F}$ with {$\mathcal{E}_1^{\frac{1}{2}}$-norm} and dense in $C_\mathrm{c}(E)$ with uniform norm. A \emph{core} of $\mathcal{E}$ is by definition a subset $\mathcal{C}$ of $\mathcal{F}\cap C_\mathrm{c}(E)$ such that $\mathcal{C}$ is dense in $\mathcal{F}$ with {$\mathcal{E}_1^{\frac{1}{2}}$-norm} and dense in $C_\mathrm{c}(E)$ with uniform norm.  We refer more terminologies of Dirichlet forms, such as \emph{standard core}, \emph{special standard core}, \emph{recurrence}, \emph{transience}, \emph{irreducibility} and etc, to \cite{CM} and \cite{FU}. 

\begin{definition}
Let $(\mathcal{E,F})$ and $(\mathcal{E}',\mathcal{F}')$ be two regular Dirichlet forms on $L^2(E,m)$. Then $(\mathcal{E}',\mathcal{F}')$ is said to be a \emph{regular subspace} of $(\mathcal{E,F})$, denoted by 
\[
(\mathcal{E',F'})\prec (\mathcal{E,F})\text{ or }\mathcal{E}'\prec \mathcal{E},
\]
 if 
\begin{equation}
 \mathcal{F'} \subset \mathcal{F}, \qquad \mathcal{E}(u,v)=\mathcal{E'}(u,v) \quad  {u,v \in \mathcal{F'}}.
\end{equation}
If in addition, $\mathcal{F}'$ is a proper subset of $\mathcal{F}$, then $(\mathcal{E',F'})$ is said to be a \emph{proper regular subspace} of $(\mathcal{E,F})$. If $X$ and $X'$ are the associated Markov processes of $(\mathcal{E,F})$ and $(\mathcal{E}',\mathcal{F}')$, we also write $$X'\prec X$$ to stand for $(\mathcal{E',F'})\prec (\mathcal{E,F})$.
\end{definition}

For a fixed regular Dirichlet form, the main concerned problems on this topic are the existence of proper regular subspaces and how to characterize them if exist. The case of one-dimensional diffusions was discussed in \cite{FMG}, \cite{XPJ} and \cite{LY2}, which indicate that all proper regular subspaces of a fixed one-dimensional diffusion can be characterized by a class of scale functions, see Theorem~4.1 of  \cite{XPJ}. In particular, in our another article \cite{LY2}, we considered the traces of one-dimensional Brownian motion and its regular subspace on the boundary of a characteristic set and found that they consist of a beautiful Beurling-Deny type decomposition. Furthermore, the regular subspaces of general Dirichlet forms were studied in  \cite{LY}, and its main results imply that the structure of regular subspaces is invariant under several probabilistic transforms, such as spatial homeomorphous transforms, time changes with full quasi support, killing and resurrected transforms, see Theorem~2 of  \cite{LY}. Moreover, every regular subspace inherits the jumping and killing measures of original Dirichlet form. As a sequel, the regular subspaces of a local Dirichlet form are still local.

What we are concerned in this paper are the regular subspaces of \emph{skew product diffusions}. The skew product of two diffusions was first raised by K. It\^o and H.P. McKean in  \cite{KH}. Then A.R. Galmarino \cite{GAR} proved that every isotropic diffusion $X$ on $\mathbf{R}^3$ can be expressed as the following skew product of radial process $(r_t)_{t\geq 0}$ and an independent spherical Brownian motion $\vartheta$ run with a clock $(A_t)_{t\geq 0}$ depending on the radial path, i.e.
\begin{equation}\label{EQRI}
	X_t=(r_t,\vartheta_{A_t}),\quad t\geq 0.
\end{equation}
More precisely, $A$ is a positive continuous additive functional (PCAF in abbreviation) of $r$. The expressions of Dirichlet forms associated with the skew products of symmetric diffusion processes were first constructed by M. Fukushima and Y. Oshima in  \cite{MY} for some special cases on the smooth manifolds. Then H. \^Okura  \cite{OH} extended the results of \cite{MY} to general cases. Their conclusions depend on the conservativeness of two independent diffusion components. { But then, H. \^Okura in \cite{OH97} cancelled the conservativeness of $X^1$ and $X^2$. On the other hand,}  we have already employed a special case of skew product method, say the direct product, to discuss the regular subspaces of high-dimensional Brownian motions in our previous paper \cite{LY}. Thus some results in this paper, for example Theorem~\ref{THM2}, may be regarded as the extensions of some results in \cite{LY}.   

The structure of this paper is as follows. In \S\ref{SEC2}, we shall make a brief introduction to the associated Dirichlet forms of skew product diffusions and their global properties. Note that a skew product diffusion $X=(X_t)_{t\geq 0}$ can be written as
\[
	X_t:=(X^1_t,X^2_{A_t}),\quad t\geq 0,
\]
where $X^1$ and $X^2$ are two independent symmetric diffusions, and $A$ is a PCAF of $X^1$. {Thus the global properties of $X$ may have some connections with those of $X^1$ and $X^2$. Particularly, we shall also give some examples, say rotationally invariant diffusions on $\mathbf{R}^d\setminus \{0\}$, to illustrate  their associated Dirichlet forms and the relevant global properties. }

In \S\ref{RES}, our main purpose is to characterize the skew product type regular subspaces of $X$  in Theorem~\ref{THM2}. More precisely, for any skew product type regular subspace of $X$, say 
\[
	Y_t=(Y^1_t,Y^2_{\tilde{A}_t}),\quad t\geq 0,
\]
we obtain that the PCAFs $A$ and $\tilde{A}$ share the common associated smooth measure, and $Y^i$ corresponds to a regular subspace of $X^i$ for $i=1,2$.  

In particular, we shall consider the rotationally invariant diffusions and their regular subspaces on the Euclidean spaces. As noted above, these diffusions may be written as the form of skew products. In some cases, such as Example~\ref{EXA1} and \ref{EXA2}, we can give their part Dirichlet forms on the state space $\mathbf{R}^d\setminus\{0\}=(0,\infty)\times S^{d-1}$ by using Theorem~\ref{THM1}. But it will become very tough if we replace the state space by the whole Euclidean space  $\mathbf{R}^d$, since their ``skew" smooth measures are not Radon on $[0,\infty)$, and Theorem~\ref{THM1} is not valid any more. However, the rotationally invariant diffusions usually have an expression of energy form, which is also described as a distorted Brownian motion in \cite{AHS3}, on the whole Euclidean space. Thus we might do well to extend our discussions from $\mathbf{R}^d\setminus \{0\}$ to $\mathbf{R}^d$. Fortunately, the second author and his co-authors introduced an idea of one-point extension for Markov processes in \cite{CM2}, \cite{CFY}, \cite{MF3} and \cite{FT}. They glued an additional motion around the ``one-point" boundary (via the entrance law) and the original process together to produce a new extended process on the whole state space. In \S\ref{RID}, we shall also construct a similar extension from $\mathbf{R}^d\setminus \{0\}$ to $\mathbf{R}^d$ for the regular subspace of rotationally invariant diffusion. More precisely, let $(\mathcal{E,F})$ be a regular Dirichlet form on $L^2(\mathbf{R}^d,m)$, where $m$ is Radon on $\mathbf{R}^d$ and absolutely continuous with respect to the Lebesgue measure. We still denote the restricted measure of $m$ on $\mathbf{R}^d\setminus \{0\}$ by $m$. The part Dirichlet form $(\mathcal{E}^0,\mathcal{F}^0)$ of $(\mathcal{E,F})$ on $L^2(\mathbf{R}^d\setminus \{0\},m)$ is given by 
\[
\begin{aligned}
	&\mathcal{F}^0:=\{u\in \mathcal{F}:\tilde{u}(0)=0,\quad \mathcal{E}\text{-q.e.}\},\\
	&\mathcal{E}^0(u,v):=\mathcal{E}(u,v),\quad u,v\in \mathcal{F}^0,
\end{aligned}
\]
 and it is regular on $L^2(\mathbf{R}^d\setminus \{0\},m)$. Actually, the part Dirichlet form on $\mathbf{R}^d\setminus \{0\}$ of a regular subspace of $(\mathcal{E,F})$ is still a regular subspace of $(\mathcal{E}^0,\mathcal{F}^0)$. We want to explore the inverse question. In other words, for a fixed regular subspace $(\tilde{\mathcal{E}}^0,\tilde{\mathcal{F}}^0)$ of $(\mathcal{E}^0,\mathcal{F}^0)$ on $L^2(\mathbf{R}^d\setminus \{0\},m)$, i.e. 
 \[
 	(\tilde{\mathcal{E}}^0,\tilde{\mathcal{F}}^0)\prec (\mathcal{E}^0,\mathcal{F}^0),
 \]
we want to find a regular subspace $(\tilde{\mathcal{E}},\tilde{\mathcal{F}})$ of $(\mathcal{E,F})$ on $\mathbf{R}^d$, i.e. 
\[
	(\tilde{\mathcal{E}},\tilde{\mathcal{F}}) \prec (\mathcal{E,F}),
\]
such that $(\tilde{\mathcal{E}}^0,\tilde{\mathcal{F}}^0)$ is exactly the part Dirichlet form of $(\tilde{\mathcal{E}},\tilde{\mathcal{F}})$ on $\mathbf{R}^d\setminus \{0\}$. If $(\tilde{\mathcal{E}},\tilde{\mathcal{F}})$ exists, we call it the regular extension of $(\tilde{\mathcal{E}}^0,\tilde{\mathcal{F}}^0)$ on $\mathbf{R}^d$. In Theorem~\ref{THM3}, we shall prove the existence and uniqueness of regular  extension for a fixed regular subspace in the context of rotationally invariant diffusion. In particular,  even if the original Dirichlet form $(\mathcal{E},\mathcal{F})$ and its part Dirichlet form $(\mathcal{E}^0,\mathcal{F}^0)$ are equal, the regular subspace $(\tilde{\mathcal{E}}^0,\tilde{\mathcal{F}}^0)$ and its regular extension $(\tilde{\mathcal{E}},\tilde{\mathcal{F}})$ may still be different, see Corollary~\ref{COR6}. Note that the regular extension appeared in \S\ref{RID} is a special one-point extension outlined above.

\section{The Dirichlet forms of skew product diffusions and their global properties}\label{SEC2}

Let $E_1,E_2$ be two locally compact separable metric spaces and $m_1,m_2$ two Radon measures fully supported on $E_1,E_2$ respectively. Assume $(\mathcal{E}^1,\mathcal{F}^1)$ and $(\mathcal{E}^2,\mathcal{F}^2)$ are two regular strongly local Dirichlet forms on $L^2(E_1,m_1)$ and $L^2(E_2,m_2)$ respectively. Clearly 
\[
	\mathcal{C}_1:=\mathcal{F}^1\cap C_c(E_1),\quad \mathcal{C}_2:=\mathcal{F}^2\cap C_c(E_2)
\]
are their special and standard cores. Denote the Markov processes associated with $(\mathcal{E}^1,\mathcal{F}^1)$ and $(\mathcal{E}^2,\mathcal{F}^2)$ by $X^1$ and $X^2$. Then $X^1$ and $X^2$ are two symmetric diffusions without killing inside on $E_1$ and $E_2$. Without loss of generality, we assume that $X^1$ and $X^2$ are independent. Given the linear spaces $\mathcal{D}_1$ and $\mathcal{D}_2$ of functions on $E_1$ and $E_2$ respectively, we denote the \emph{tensor product} of $\mathcal{D}_1$ and $\mathcal{D}_2$ by $\mathcal{D}_1\otimes \mathcal{D}_2$. More precisely, $\mathcal{D}_1\otimes \mathcal{D}_2$ is the linear space generated by $$\{u\otimes v:u\in \mathcal{D}_1,v\in \mathcal{D}_2\},$$ where $u\otimes v(x_1,x_2)=u(x_1)v(x_2)$ for any $x_1\in E_1,x_2\in E_2$. Moreover, let
\[
	E:=E_1\times E_2,\quad m:=m_1\times m_2
\]
and 
\[
\mathcal{C}:=\mathcal{C}_1\otimes \mathcal{C}_2,
\]
i.e. the tensor product of $\mathcal{C}_1$ and $\mathcal{C}_2$. For a PCAF $A$ of $X^1$, denote its \emph{Revuz measure} by $\mu_A$. 
The following Markov process on $E$
\begin{equation}\label{EQ2}
 X_t:=(X^1_t,X^2_{A_t}),\quad t\geq 0,
\end{equation}
which is referred to \cite{MY}, 
is called \emph{the skew product of $X^1$ and $X^2$ with respect to $A$}, and we denote it by 
\[
	X=[X^1,X^2,A] \text{ or } [X^1,X^2,\mu_A].
\]
Clearly $X$ is a diffusion process on $E$. So we also call $X$ \emph{a skew product diffusion} if it makes no confusion. The $L^2$-semigroups of $X^1,X^2,X$ are denoted by $\{T^1_t\},\{T^2_t\},\{T_t\}$ respectively. 

On the other hand, a Markov process $Y$ is said to be conservative if its probability semigroup $(P^Y_t)_{t\geq 0}$ satisfies that $P^Y_t1\equiv 1$ for any $t\geq 0$. If in addition, $Y$ is symmetric, then $Y$ is conservative if and only if the symmetric measure is an invariant measure. For any measure space $(M,\lambda)$ and real Hilbert space $\mathbf{H}$, we denote by $L^2(M,\lambda;\mathbf{H})$ the real $L^2$-space of $\mathbf{H}$-value functions on $M$. {The following theorem is taken from \cite{OH97} for handy reference}.

\begin{theorem}\label{THM1}
{If  $\mu_A$ is Radon}, then $X$ is $m$-symmetric on $E$. Furthermore, the associated Dirichlet form $(\mathcal{E,F})$ of $X$ is regular on $L^2(E,m)$   and admits the following expression: for any $u\in \mathcal{F}\cap C_c(E)$, 
\begin{equation}\label{EQUEF}
\begin{aligned}
	&[x_2\rightarrow u(\cdot, x_2)\in \mathcal{F}^1]\in L^2(E_2,m_2;\mathcal{F}^1),\\
	&[x_1\rightarrow u(x_1, \cdot)\in \mathcal{F}^2]\in L^2(E_1,\mu_A;\mathcal{F}^2)
\end{aligned}\end{equation}
and 
\begin{equation}\label{EQFGMA}
	\mathcal{E}(u,u)=\int_{E_2}\mathcal{E}^1(u(\cdot, y),u(\cdot,y))m_2(dy)+\int_{E_1}\mathcal{E}^2(u(x,\cdot),u(x,\cdot))\mu_A(dx).
\end{equation}
If in particular, $\mu_A\leq Cm_1$ for some constant $C>0$, then  $\mathcal{C}$ is a core of $(\mathcal{E,F})$, and \eqref{EQUEF},  \eqref{EQFGMA} hold for any $u\in \mathcal{F}$.
\end{theorem}

{ The Dirichlet forms of skew product diffusions were first constructed for a special case, say $X^1, X^2$ are $C_c^\infty$-regular, in  \cite{MY}. More precisely, $E_1,E_2$ are $C^\infty$-differentiable manifolds, and 
	\[
		\mathcal{C}_i=C_c^\infty(E_i),
	\]
i.e. the class of all infinitely differentiable functions with compact support on $E_i$ for $i=1,2$. The other settings are the same as above. Then the skew product diffusion \eqref{EQ2} is $m$-symmetric and its associated Dirichlet form is regular with a core $C_c^\infty(E)$. After that, H. \^Okura \cite{OH} extended the results of \cite{MY} to general cases, and in \cite{OH97}, he omitted the conservativeness of $X^1$ and $X^2$. 

Our main concerns are the regular subspaces of $(\mathcal{E},\mathcal{F})$. 
Since the existence and expression of regular subspaces depend on the global properties of original Dirichlet form, we are now in position to present some global properties of $X$ for handy reference, which are taken from \cite{MY} \cite{OH} and \cite{OH97}.

\begin{proposition}\label{PRO}
	Under the same assumptions of Theorem~\ref{THM1}, the following assertions are right:
	\begin{itemize}
	\item[(1)]  if $(\mathcal{E}^1,\mathcal{F}^1)$ and $(\mathcal{E}^2,\mathcal{F}^2)$ are both recurrent, and either $\mu_A(E_1)<\infty$ or $m_2(E_2)<\infty$, then $(\mathcal{E,F})$ is recurrent;
	\item[(2)] if $(\mathcal{E}^1,\mathcal{F}^1)$ is transient,  then $(\mathcal{E,F})$ is also transient; 
	\item[(3)] $(\mathcal{E,F})$ is irreducible if and only if $(\mathcal{E}^1,\mathcal{F}^1)$ and $(\mathcal{E}^2,\mathcal{F}^2)$ are both irreducible.
	\end{itemize}
\end{proposition}
}

As outlined in \S\ref{intro}, every rotationally invariant diffusion on $\mathbf{R}^d\setminus\{0\}$ has a representation \eqref{EQRI}, and we also denote it by
\[
	X=(r_t,\vartheta_{A_t})_{t\geq 0}:=[r, A].
\]
If the associated Revuz measure of $A$ is $\mu$, we also write
\begin{equation}\label{EQXRM}
	X=[r,\mu].
\end{equation}
Note that this is a speical case of skew product diffusion, i.e. 
\[
	E_1=(0,\infty),\quad E_2=S^{d-1},\quad X^1=r, \quad X^2=\vartheta.
\]
We refer the minimal diffusion to Example~3.5.7 of  \cite{CM}.
Let $\mathcal{R}$ be a class of rotationally invariant diffusions as follows
\begin{equation}
	\mathcal{R}:=\big\{X=[r,\mu]: r\text{ is a minimal diffusion, and }\mu \text{ is Radon on }(0,\infty)\big\}.
\end{equation}

\begin{remark}\label{RM3}
 The purpose of this note is to illustrate the global properties of minimal diffusion $r$. All minimal diffusions are irreducible. Denote the scale function and speed measure  of $r$  by $p$ and $l$. Then $r$ is transient if and only if $0$ or $\infty$ is approachable relative to $r$, i.e. $p(0+):=\lim_{x\downarrow 0}p(x)>-\infty$ or $p(\infty):=\lim_{x\uparrow \infty}p(x)<\infty$. Otherwise, $r$ is recurrent. We refer the definition of \emph{approachable boundary point in finite time} of $r$ to Example~3.5.7 of  \cite{CM}. It is well known that $0$ (resp. $\infty$) is approachable in finite time if and only if for some constant $c\in (0,\infty)$, 
\[
	\int_0^c l((x,c))dp(x)<\infty,\quad (\text{resp.} \int_c^\infty l((c,x))dp(x)<\infty).
\]
Clearly every regular boundary point is approachable in finite time. Note that $r$ is conservative if and only if neither $0$ nor $\infty$ is approachable in finite time.
\end{remark}
 
We can characterize the global properties of \eqref{EQXRM} as follows.  

\begin{proposition}\label{COR5}
	Let $X=[r,\mu]\in\mathcal{R}$. Then $X$ is always irreducible. Furthermore, $X$ is recurrent (resp. transient, conservative) if and only if $r$ is recurrent (resp. transient, conservative). 
\end{proposition}
\begin{proof}
The irreducibility of $X$ follows from {Proposition~\ref{PRO} (3)} and the irreducibility of $r$. Since $r$ is the radius part of $X$, i.e. $r=|X|$, it follows that $X$ is conservative if and only if $r$ is conservative.

If $r$ is transient, then it follows from {Proposition~\ref{PRO} (2)} that $X$ is transient. If $r$ is recurrent, it follows from {Proposition~\ref{PRO} (1)} that $X$ is also recurrent. Note that the irreducibility of $X$ (resp. $r$) implies that $X$ (resp. $r$) is recurrent or transient. That completes the proof.  
\end{proof}

The two examples below are two typical examples of skew product diffusions. In fact, they both belong to $\mathcal{R}$ for a fixed number $d\geq 2$. We shall also analyse their global properties by using the above conclusions.

\begin{example}\label{EXA1}
	Fix a natural number $d\geq 2$. Let 
	\[
		E_1=(0,\infty),\quad E_2=S^{d-1}:=\{x\in \mathbf{R}^d: |x|=1\}.
	\]
	 Assume that $(r_t)_{t\geq 0}$ is the Bessel process on $(0,\infty)$ with scale function $p$ and  speed measure $x^{d-1}dx$, where $p(x)=\log x$ when $d=2$ and $1/(2-d)\cdot x^{2-d}$ when $d\geq 3$. Let $\vartheta$ be the spherical Brownian motion on $S^{d-1}$, which corresponds to the \emph{Beltrami-Laplace operator} $\Delta_{S^{d-1}}$. Further set 
	\[
		A_t:=\int_0^t \frac{1}{r_s^2}ds,\quad t\geq 0.
	\]
	Then the skew product diffusion $$B=(r_t, \vartheta_{A_t})_{t\geq 0}$$ is exactly the $d$-dimensional Brownian motion on $\mathbf{R}^d\setminus \{0\}$. Notice that $\{0\}$ is polar relative to  $d$-dimensional Brownian motion. It is well known that $\vartheta$ is always irreducible, recurrent and hence conservative. The Bessel process $r$ is always irreducible and it is  recurrent when $d\leq 2$. Thus  from {Proposition~\ref{PRO} (1) and (3)}, we obtain that the $d$-dimensional Brownian motion is irreducible and recurrent when $d=2$. Otherwise, it is irreducible and transient. 
\end{example}

\begin{example}\label{EXA2}
 	Let 
 \[
 	E_1=(0,\infty),\quad E_2=S^2=\{x\in\mathbf{R}^3:|x|=1\}.
 	\]
  Fix a constant $\gamma>0$. A symmetric diffusion $(r_t)_{t\geq 0}$ on $(0,\infty)$ killed upon hitting $\{0\}$ is characterized by its scale function $s$ and speed measure $l$:
 	\[
 \begin{aligned}
 	&s(x)=\frac{1}{4\gamma^2}\text{e}^{2\gamma x}, \\&l(dx)=2\gamma \text{e}^{-2\gamma x}dx.
 \end{aligned}\]
Note that $l$ is a finite measure on $(0,\infty)$. Thus $0$ is a \emph{regular boundary point} of $(r_t)_{t\geq 0}$, and its another boundary point $\infty$ is \emph{unapproachable} (Cf. \S2.2.3 of  \cite{CM}). As a result,  $(r_t)_{t\geq 0}$ is irreducible and transient, and $\{0\}$ is not polar with respect to  $(r_t)_{t\geq 0}$. On the other hand, similar to  Example~\ref{EXA1}, let $\vartheta$ be the spherical Brownian motion on $S^2$ and 
 \[
		A_t:=\int_0^t \frac{1}{r_s^2}ds,\quad t\geq 0.
	\]
Then the skew product diffusion 
\[
	X^0:=(r_t,\vartheta_{A_t})_{t\geq 0}
	\]
 is a rotationally invariant diffusion on $\mathbf{R}^3\setminus \{\mathbf{0}\}$  killed upon hitting $\{\mathbf{0}\}$. Since $r$ is irreducible and transient, it follows that $X^0$ is also irreducible and transient. 

In fact, $X^0$ is  the part process on $\mathbf{R}^3\setminus \{\mathbf{0}\}$  of a diffusion $X$ on $\mathbf{R}^3$ induced by the following energy form on $L^2(\mathbf{R}^3,\psi_\gamma^2dx)$:
\begin{equation}\label{EQFEF}
	\begin{aligned}
		&\mathcal{F}:=\{u\in L^2(\mathbf{R}^3,\psi_\gamma^2dx): \nabla u\in L^2(\mathbf{R}^3,\psi_\gamma^2dx)\}\\
		&\mathcal{E}(u,v):=\frac{1}{2}\int_{\mathbf{R}^3} \nabla u(x)\cdot \nabla v(x) \psi_\gamma(x)^2dx,\quad u,v\in \mathcal{F},
	\end{aligned}
\end{equation}
 where $\nabla u$ is the gradient of $u$ in  sense of weak distribution and
\begin{equation}\label{EQPHI}
\psi_\gamma(x):=\frac{\sqrt{\gamma}}{\sqrt{2\pi}}\frac{\text{e}^{-\gamma |x|}}{|x|}.
\end{equation}
In particular, $C_c^\infty(\mathbf{R}^3)$ is a core of $(\mathcal{E,F})$, and $\{\mathbf{0}\}$ is not $\psi_\gamma^2dx$-polar with respect to $X$. As a sequel, the potential behaviour of $X$ outside $\{\mathbf{0}\}$, i.e. its part process $X^0$, is very similar to  3-dimensional Brownian motion, whereas it becomes very wired upon approaching $\{0\}$.  We refer more details to  \cite{FL2}.
\end{example}

\section{The regular subspaces of skew product diffusions}\label{RES}

Our main purpose of this section is to characterize the regular subspaces of skew product diffusions outlined in Theorem~\ref{THM1}. Before that, we need to give some more notations. 
Denote all $m_i$-symmetric diffusions on $E^i$, whose associated Dirichlet forms are regular and strongly local on $L^2(E_i,m_i)$, by $\mathcal{D}(E_i)$ for $i=1,2$. Given $X^1\in \mathcal{D}(E_1)$, $\text{S}_{E_1}(X^1)$ represents all Radon smooth measures on $E_1$ with respect to $X^1$. Define the class of all  $m$-symmetric skew product diffusions on $E$ characterized in Theorem~\ref{THM1} by
\begin{equation}
	\mathcal{D}_{\mathrm{sp}}(E):= \bigg\{ X=[X^1,X^2,\mu]:X^i\in \mathcal{D}(E_i),i=1,2; \mu \in \text{S}_{E_1}(X^1) \bigg\}.
\end{equation}
For any $X\in \mathcal{D}_{\mathrm{sp}}(E)$, we also write $X\in \mathcal{D}_{\mathrm{sp}}$ since the state space is fixed. If  $(\mathcal{E,F})$ is the associated Dirichlet form   of $X$, we also write $(\mathcal{E,F})\in \mathcal{D}_{\mathrm{sp}}$ or $\mathcal{E}\in \mathcal{D}_{\mathrm{sp}}$ to represent $X\in \mathcal{D}_{\mathrm{sp}}$ for convenience. Recall that for two regular Dirichlet forms $(\mathcal{E,F})$ and $(\mathcal{E',F'})$ on a same Hilbert space, $(\mathcal{E',F'})\prec (\mathcal{E,F})$ means
\[
	\mathcal{F}'\subset \mathcal{F},\quad \mathcal{E}(u,v)=\mathcal{E}'(u,v)\quad u,v\in \mathcal{F}'.
\]
Fixing a constant $c>0$, we say $(\mathcal{E',F'})$ is \emph{a regular subspace of} $(\mathcal{E,F})$ \emph{up to the constant} $c$, denoted by 
\[
	(\mathcal{E',F'})\prec_c (\mathcal{E,F}),
\]
if
\[
	\mathcal{F}'\subset \mathcal{F},\quad \mathcal{E}(u,v)=c\cdot \mathcal{E}'(u,v),\quad u,v\in \mathcal{F}'.
\]
When $c=1$, the notation $``\prec_c"$ is exactly the same as $``\prec"$.
Our main result of this section is as follows. The special case $\mu=m_1$ of this theorem has already been discussed in  \cite{LY}.

\begin{theorem}\label{THM2}
	Let $X=[X^1,X^2,\mu],\tilde{X}=[\tilde{X}^1,\tilde{X}^2,\tilde{\mu}]\in \mathcal{D}_{\mathrm{sp}}$. The associated Dirichlet  forms of $X^1,X^2,X,\tilde{X}^1,\tilde{X}^2,\tilde{X}$ are denoted by $(\mathcal{E}^1,\mathcal{F}^1)$, $(\mathcal{E}^2,\mathcal{F}^2)$, $(\mathcal{E},\mathcal{F})$, $(\tilde{\mathcal{E}}^1,\tilde{\mathcal{F}}^1)$, $(\tilde{\mathcal{E}}^2,\tilde{\mathcal{F}}^2)$, $(\tilde{\mathcal{E}},\tilde{\mathcal{F}})$ on appropriate Hilbert spaces $L^2(E_1,m_1)$, $L^2(E_2,m_2)$ or $L^2(E,m)$. Then $(\tilde{\mathcal{E}},\tilde{\mathcal{F}})\prec (\mathcal{E},\mathcal{F})$ if and only if there exists a constant $c>0$ such that
\begin{equation}\label{EQMEF}
\tilde{\mu}=c\cdot \mu,\quad (\tilde{\mathcal{E}}^1,\tilde{\mathcal{F}}^1)\prec (\mathcal{E}^1,\mathcal{F}^1),\quad (\tilde{\mathcal{E}}^2,\tilde{\mathcal{F}}^2)\prec_c (\mathcal{E}^2,\mathcal{F}^2).
\end{equation}
In particular, $(\tilde{\mathcal{E}},\tilde{\mathcal{F}})$ is a proper regular subspace of $(\mathcal{E,F})$ if and only if in addition $\tilde{\mathcal{F}}^1\neq \mathcal{F}^1$ or $\tilde{\mathcal{F}}^2\neq \mathcal{F}^2$.
\end{theorem}
\begin{proof}
First assume that there exists a constant $c>0$ such that \eqref{EQMEF} holds.  Let
\[
\tilde{\mathcal{C}}_1:=\tilde{\mathcal{F}}^1\cap C_c(E_1),\quad \tilde{\mathcal{C}}_2:=\tilde{\mathcal{F}}^1\cap C_c(E_2),\quad \tilde{\mathcal{C}}:=\tilde{\mathcal{C}}_1\otimes \tilde{\mathcal{C}}_2.
\]
Since the time-change transform relative to a PCAF with full quasi support  remains the inclusion relationship of Dirichlet spaces (Cf. \S2.2.2 of \cite{LY}), it follows from the proof of Theorem~1.3 of \cite{OH} that we only need to prove $(\tilde{\mathcal{E}},\tilde{\mathcal{F}})\prec (\mathcal{E},\mathcal{F})$ for the special case $\mu\leq m_1$. In particular, $\mathcal{C}$ is a core of $(\mathcal{E,F})$, and $\tilde{\mathcal{C}}$ is a core of $(\tilde{\mathcal{E}},\tilde{\mathcal{F}})$.

Take any two functions $u=f_1\otimes f_2,v=g_1\otimes g_2\in \tilde{\mathcal{C}}$, where $f_1,g_1\in \tilde{\mathcal{C}}_1$ and $f_2,g_2\in \tilde{\mathcal{C}}_2$. It follows from \eqref{EQMEF} and Theorem~\ref{THM1}  that
\[\begin{aligned}
	\tilde{\mathcal{E}}(u,v)&=\tilde{\mathcal{E}}^1(f_1,g_1)\int f_2\cdot g_2dm_2+\tilde{\mathcal{E}}^2(f_2,g_2)\int f_1\cdot g_1 d\tilde{\mu}\\
	&=\mathcal{E}^1(f_1,g_1)\int f_2\cdot g_2dm_2 +\frac{1}{c}\cdot \mathcal{E}^2(f_2,g_2)\cdot c\cdot \int f_1\cdot g_1 d\mu \\
	&=\mathcal{E}(u,v).
\end{aligned}\]
The above equality is also right for any $u,v\in \tilde{\mathcal{C}}$. Since $\tilde{\mathcal{C}}$ is a core of $(\tilde{\mathcal{E}},\tilde{\mathcal{F}})$, we can conclude that $\tilde{\mathcal{F}}\subset \mathcal{F}$ and $\mathcal{E}(u,v)=\tilde{\mathcal{E}}(u,v)$ for any $u,v\in \tilde{\mathcal{F}}$.

On the contrary, assume $(\tilde{\mathcal{E}},\tilde{\mathcal{F}})\prec (\mathcal{E,F})$.
 We assert that 
\begin{equation}\label{EQCC}
	\tilde{\mathcal{C}}_1\subset \mathcal{C}_1,\quad \tilde{\mathcal{C}}_2\subset \mathcal{C}_2.
\end{equation}
In fact, take two functions $f\in \tilde{\mathcal{C}}_1,g\in \tilde{\mathcal{C}}_2$, then $u:=f\otimes g\in \tilde{\mathcal{F}}\subset \mathcal{F}$. Clearly $u\in C_c(E)$. Hence
\begin{equation}\label{EQUC}
	u\in \mathcal{F}\cap C_c(E).
\end{equation}
It follows from \eqref{EQUEF} that $f\in \mathcal{F}^1,g\in\mathcal{F}^2$, and thus \eqref{EQCC} holds.

Next, we still take $f\in \tilde{\mathcal{C}}_1,g\in \tilde{\mathcal{C}}_2$ and let $u=f\otimes g$. It follows from Theorem~\ref{THM1} and $\mathcal{E}(u,u)=\tilde{\mathcal{E}}(u,u)$ that
\[
\begin{aligned}
	&\mathcal{E}^1(f,f)\cdot \int g^2dm_2+\mathcal{E}^2(g,g)\cdot\int f^2 d\mu\\=&\tilde{\mathcal{E}}^1(f,f)\cdot \int g^2dm_2+\tilde{\mathcal{E}}^2(g,g)\cdot\int f^2 d\tilde{\mu}.
\end{aligned}\]
Fix the function $f$. Denote $$a_f:=\int f^2d\mu,\quad \tilde{a}_f:=\int f^2d\tilde{\mu}$$ and $$b_f:=\mathcal{E}^1(f,f),\quad \tilde{b}_f:=\tilde{\mathcal{E}}^1(f,f).$$ Without loss of generality, assume $a_f,\tilde{a}_f>0$ and $b_f\leq \tilde{b}_f$. Then we can conclude that
\begin{equation}\label{EQEFG}
	\mathcal{E}^2(g,g)=\frac{\tilde{a}_f}{a_f}\cdot \tilde{\mathcal{E}}^2(g,g)+\frac{\tilde{b}_f-b_f}{a_f}\cdot \int g^2 dm_2
\end{equation}
for any $g\in \tilde{\mathcal{C}}_2$. Since $(\mathcal{E}^2,\mathcal{F}^2)$ and $(\tilde{\mathcal{E}}^2,\tilde{\mathcal{F}}^2)$ are strongly local, it follows that $\tilde{b}_f=b_f$. In other words, $\mathcal{E}^1(f,f)=\tilde{\mathcal{E}}^1(f,f)$ for any $f\in \tilde{\mathcal{C}}_1$. Therefore
\[
	(\tilde{\mathcal{E}}^1,\tilde{\mathcal{F}}^1)\prec (\mathcal{E}^1,\mathcal{F}^1).
\]
Moreover, since any $g\in \tilde{\mathcal{C}}_2$ satisfies \eqref{EQEFG}, we can deduce that there exists a constant $c>0$ such that for any $f\in \tilde{\mathcal{C}}_1$ with $a_f,\tilde{a}_f>0$, it holds that $\tilde{a}_f/a_f=c$, i.e.
\[
	\int f^2d\tilde{\mu}=c\cdot \int f^2d\mu.
\]
Since $\tilde{\mathcal{C}}_1$ is dense in $C_c(E_1)$, it follows that
\begin{equation}\label{EQMCM}
	\tilde{\mu}=c\cdot \mu.
\end{equation}
Consequently, $\mathcal{E}^2(g,g)=c\cdot \tilde{\mathcal{E}}^2(g,g)$ for any $g\in \tilde{\mathcal{C}}_2$. Then we obtain that
\[
	(\tilde{\mathcal{E}}^2,\tilde{\mathcal{F}}^2)\prec_c (\mathcal{E}^2,\mathcal{F}^2).
\]
The second assertion about the proper property of regular subspaces is obvious. That completes the proof.  
\end{proof}

\begin{remark}
	Note that if $X^1,\tilde{X}^1\in \mathcal{D}(E_1)$ and the associated Dirichlet form $(\tilde{\mathcal{E}}^1,\tilde{\mathcal{F}}^1)$ of $\tilde{X}^1$ is a regular subspace of associated Dirichlet form $(\mathcal{E}^1,\mathcal{F}^1)$ of $X^1$, then $\mu\in \text{S}_{E_1}(X^1)$ is always a smooth measure relative to $\tilde{X}^1$ but not vice versa (Cf. Lemma~2 of  \cite{LY}), i.e. 
	\[
		\text{S}_{E_1}(X^1)\subset \text{S}_{E_1}(\tilde{X}^1).
	\]
This is why \eqref{EQMCM} is reasonable in  Theorem~\ref{THM2}. However, the corresponding PCAFs of $\mu$ with respect to $X^1$ and $\tilde{X}^1$ may be different. 
\end{remark}

We want to illustrate that the constant $c$ in \eqref{EQMEF} is not essential. More precisely, we have the following corollary of Theorem~\ref{THM2}. 

\begin{corollary}
	Let the notations and conditions be the same as Theorem~\ref{THM2}. Then there is an equivalent representation of $\tilde{X}$, say $\tilde{X}=[Y^1,Y^2,\mu]$, such that 
	\[
		Y^1\prec X^1,\quad Y^2\prec X^2.
	\]
\end{corollary}
\begin{proof}
Let $(\tilde{A}_t)_{t\geq 0}$ be the associated PCAF of $\mu$ with respect to $\tilde{X}^1$, and define 
\[
	\tilde{X}^{2,c}_t:=\tilde{X}^2_{ct},\quad  t\geq 0.
\]
Then the associated PCAF of $\tilde{\mu}=c\cdot \mu$ with respect to $\tilde{X}^1$ is exactly $(c\cdot \tilde{A}_t)_{t\geq 0}$. Furthermore,
\[
	[\tilde{X}^1,\tilde{X}^2,c\cdot \mu]=(\tilde{X}^1_t,\tilde{X}^2_{c \tilde{A}_t})_{t\geq 0}=(\tilde{X}^1_t,\tilde{X}^{2,c}_{\tilde{A}_t})_{t\geq 0}=[\tilde{X}^1,\tilde{X}^{2,c},\mu].
\] 
Clearly we can conclude $\tilde{X}^{2,c}\prec X^2$. Set $Y^1:=\tilde{X}^1$ and $Y^2:=\tilde{X}^{2,c}$. That completes the proof. 
\end{proof}

Now we are going to explore the regular subspaces of rotationally invariant diffusions, which belong to $\mathcal{R}$ for a fixed natural number $d\geq 2$. 
The following corollary is a direct result of Theorem~\ref{THM2}. It indicates that we can reduce the problems about regular subspaces of rotationally invariant diffusions to those of their radius parts, i.e. minimal diffusions on $(0,\infty)$, which have been considered in \cite{XPJ}.

\begin{corollary}\label{COR3}
 Let $X=[r,\mu]$ and $X'=[r',\mu']$ be in $ \mathcal{R}$. Then $X'\prec X$ if and only if $r'\prec r$ and $\mu=\mu'$.
\end{corollary}

The next remark contains some general conclusions about global properties of regular subspaces, which can be found in \cite{FMG} \cite{LY2} and \cite{LY}. After these notes, we shall discuss the  regular subspaces of skew product diffusions  outlined in Example~\ref{EXA1} and \ref{EXA2} and analyse their global properties.

\begin{remark}\label{RM4}
Fix two regular and strongly local Dirichlet forms, say $(\mathcal{E,F})$ and  $(\mathcal{E}',\mathcal{F}')$, on a Hilbert space $L^2(E,m)$, and assume that $$(\mathcal{E}',\mathcal{F}')\prec (\mathcal{E,F}).$$ It follows from Theorem~4.6.4 of  \cite{FU} and Remark~1 of  \cite{LY} that
\begin{itemize}
\item[(i)] if $(\mathcal{E,F})$ is irreducible, then $(\mathcal{E}',\mathcal{F}')$ is also irreducible. 
\end{itemize}
Furthermore, from Theorem~1.6.4 of  \cite{FU}, we have
\begin{itemize}
\item[(ii)] if $(\mathcal{E,F})$ is transient, then $(\mathcal{E}',\mathcal{F}')$ is also transient; 
\item[(iii)] if $(\mathcal{E}',\mathcal{F}')$ is recurrent, then $(\mathcal{E,F})$ is also recurrent.
\end{itemize}
In other words, if $(\mathcal{E,F})$ is irreducible and transient, then its regular subspaces must be irreducible and transient. If $(\mathcal{E,F})$ is irreducible and recurrent, then its regular subspaces may be irreducible and recurrent, or irreducible and transient. 
\end{remark}

\begin{example}\label{EXA3}
Fix an integer $d\geq 2$. As outlined in Example~\ref{EXA1}, the $d$-dimensional Brownian motion on $\mathbf{R}^d\setminus \{\mathbf{0}\}$ can be written as 
\[
	B=(r_t,\vartheta_{A_t})_{t\geq 0}
\]
where $r$ is the Bessel process on $(0,\infty)$ with the scale function $p$ given in Example~\ref{EXA1} and speed measure $x^{d-1}dx$, $\vartheta$ is the spherical Brownian motion on $S^{d-1}$ and 
\[
	A_t=\int_0^t \frac{1}{r^2_s}ds,\quad t\geq 0.
\]
Note that the associated Revuz measure of $A$ with respect to the speed measure of $r$ is $$\mu(dx)=x^{d-3}dx,$$ which is Radon on $(0,\infty)$ with full quasi support. Thus any $B'\in \mathcal{R}$ with $B'\prec B$ can be written as 
\[
	B'=(r'_t,\vartheta_{A'_t})_{t\geq 0},
\]
where $r'$ is a diffusion on $(0,\infty)$ with speed measure $x^{d-1}dx$ and scale function $p'$ such that $p'$ is absolutely continuous with respect to $p$,
\[
\frac{dp'}{dp}= 0\text{ or }1,\text{ a.e., (Cf. \cite{XPJ} \cite{LY2} and \cite{LY})},
\] 
 and
\[
	A'_t=\int_0^t \frac{1}{{r'}_s^2}ds.
\]
When the Lebesgue measure of $\{x\in (0,\infty):dp'/dp= 0\}$ is positive, then $B'$ is a proper regular subspace of Brownian motion $B$ on $L^2(\mathbf{R}^d\setminus\{0\})$. This gives another method to characterize the regular subspaces of high-dimensional Brownian motions, which have already been considered through direct product method in  \cite{LY}. 

When $d=2$, $r$ and $B$ are both recurrent.  In particular,  $0$ and $\infty$ are both unapproachable boundary points of $r$. Note that $l((0,c))<\infty$ and $l((c,\infty))=\infty$ for any constant $c\in (0,\infty)$. If $p'(0+)>-\infty$, equivalently $0$ is a regular boundary of $r'$, it follows from Remark~\ref{RM3} and Proposition~\ref{COR3} that $B'$ is non-conservative and transient. If $p'(0+)=-\infty$ and $p'(\infty)=\infty$, then $r'$ is recurrent, and hence $B'$ is also recurrent. If $p'(0+)=-\infty, p'(\infty)<\infty$ and $\int_c^\infty x^2p'(dx)<\infty$ for some constant $c\in (0,\infty)$, then $\infty$ is an approachable boundary point in finite time of $r'$, and $r'$ is non-conservative and transient. Hence $B'$ is also non-conservative and transient. Otherwise, i.e. $p'(0+)=-\infty, p'(\infty)<\infty$ but $\int_c^\infty x^2p'(dx)=\infty$ for some constant $c\in (0,\infty)$, $r'$ is conservative and transient. Clearly $B'$ is also conservative and transient.

When $d\geq 3$ the $d$-dimensional Brownian motion is always transient. Thus $B'$ is also transient by Remark~\ref{RM4}. Neither $0$ nor $\infty$ is approachable in finite time relative to $r$, and it implies that $r$ and $B$ are both conservative.  The conservativeness of $B'\prec B$ in $\mathcal{R}$ can be classified similarly to the case $d=2$. 
\end{example}

\begin{example}\label{EXA4}
We can also write down all rotation-invariant-type regular subspaces of $X^0$ outlined in Example~\ref{EXA2}. Note that the radius part $r^0$ of $X^0$ is transient since $0$ is a regular boudary point of $r^0$.  It follows from Proposition~\ref{COR5} and Remark~\ref{RM4} that any regular subspace of $(\mathcal{E}^0,\mathcal{F}^0)$ is transient. Moreover, since $0$ is a regular boundary point of $r$ (hence also of any regular subspace $r'\prec r$), we can conclude that $r$ (as well as $r'$) is non-conservative. Thus $X^0$ and any regular subspace $\tilde{X}^0\prec X^0$ in $\mathcal{R}$ are both non-conservative. 
\end{example}

The reason why we only discuss the Brownian motion as well as other rotationally invariant diffusions on $\mathbf{R}^d\setminus \{\mathbf{0}\}$ (not on $\mathbf{R}^d$) is that in the example of  Brownian motions, the smooth measure $$\mu(dx)=x^{d-3}dx$$ is not Radon on $[0,\infty)$ when $d=2$. Similarly the smooth measure $$l(dx)/x^2$$ in Example~\ref{EXA2} is also Radon on $(0,\infty)$ but not on $[0,\infty)$. Thus the expressions of Dirichlet forms associated with skew product diffusions outlined in Theorem~\ref{THM1} are not valid for them any more. Especially, even though they correspond to the same Brownian motion or rotationally invariant diffusion, the regular conceptions of Dirichlet forms on $\mathbf{R}^d$ and $\mathbf{R}^d\setminus \{0\}$ are different. 

\section{The regular extensions of rotation invaritant diffusions}\label{RID}

In this section, we shall construct the regular extensions of rotationally invariant diffusions from $\mathbf{R}^{d}\setminus \{0\}$ to $\mathbf{R}^d$ to overcome the difficulty appeared in the end of \S\ref{RES}. First we need to introduce the definition of regular extension.

 Let $m$ be a Radon measure on $\mathbf{R}^d$, which is absolutely continuous with respect to the Lebesgue measure. Further let $(\mathcal{E,F})$ be a regular Dirichlet form on $L^2(\mathbf{R}^d,m)$. Its part Dirichlet form on $\mathbf{R}^d\setminus \{0\}$ is denoted by $(\mathcal{E}^0,\mathcal{F}^0)$. Clearly $(\mathcal{E}^0,\mathcal{F}^0)$ is given by 
\begin{equation}\label{EQEFZ}
\begin{aligned}
	&\mathcal{F}^0:=\{u\in \mathcal{F}:\tilde{u}(0)=0,\quad \mathcal{E}\text{-q.e.}\},\\
	&\mathcal{E}^0(u,v):=\mathcal{E}(u,v),\quad u,v\in \mathcal{F}^0.
\end{aligned}
\end{equation} 
The following fact is trivial from \eqref{EQEFZ} and Remark~1 of  \cite{LY}, so we omit its proof.

\begin{lemma}\label{LM3}
	If $(\tilde{\mathcal{E}},\tilde{\mathcal{F}})\prec (\mathcal{E,F})$, then the part Dirichlet form $(\tilde{\mathcal{E}}^0,\tilde{\mathcal{F}}^0)$ on $\mathbf{R}^d\setminus \{\mathbf{0}\}$ of $(\tilde{\mathcal{E}},\tilde{\mathcal{F}})$ is a regular subspace of $(\mathcal{E}^0,\mathcal{F}^0)$, i.e. $(\tilde{\mathcal{E}}^0,\tilde{\mathcal{F}}^0)\prec (\mathcal{E}^0,\mathcal{F}^0)$.
\end{lemma}
What we are concerned is the inverse question of Lemma~\ref{LM3}. More precisely, for a fixed regular subspace $(\tilde{\mathcal{E}}^0,\tilde{\mathcal{F}}^0)\prec (\mathcal{E}^0,\mathcal{F}^0)$ on $L^2(\mathbf{R}^d\setminus \{\mathbf{0}\},m)$, whether there exists a (unique) regular subspace $(\tilde{\mathcal{E}},\tilde{\mathcal{F}})\prec (\mathcal{E,F})$ on $L^2(\mathbf{R}^d,m)$, whose part Dirichlet form on $\mathbf{R}^d\setminus \{\mathbf{0}\}$ is exactly $(\tilde{\mathcal{E}}^0,\tilde{\mathcal{F}}^0)$. If such a Dirichlet form $(\tilde{\mathcal{E}},\tilde{\mathcal{F}})$ exists, we call it the regular extension of $(\tilde{\mathcal{E}}^0,\tilde{\mathcal{F}}^0)$ on $\mathbf{R}^d$.

We shall especially consider the weighted Sobolev spaces, which were heavily studied by many authors, see \cite{AHS3}, \cite{MF2}, \cite{RW}, \cite{MZ}, \cite{MZ2} and \cite{WAD}.
Fixing a natural number $d\geq 2$, let $\rho$ be a positive and measurable function on $\mathbf{R}^d$ satisfying
\begin{equation}\label{EQRHO}
\rho\in L^1_{\text{loc}}(\mathbf{R}^d),\quad \frac{1}{\rho}\in L^1_{\text{loc}}(\mathbf{R}^d).
\end{equation}
Consider the following energy form on $L^2(\mathbf{R}^d,\rho dx)$ induced by $\rho$:
\begin{equation}
\begin{aligned}
	&\mathcal{F}:=\big\{u\in L^2(\mathbf{R}^d,\rho dx):\nabla u\in L^2(\mathbf{R}^d,\rho dx)\big \},\\
	&\mathcal{E}(u,v):=\frac{1}{2}\int_{\mathbf{R}^d}(\nabla u\cdot \nabla v)(x) \rho(x)dx,\quad u,v\in \mathcal{F},
\end{aligned}
\end{equation}
where $\nabla u$ is the gradient of $u$ in sense of weak distribution. It follows from \eqref{EQRHO} that $(\mathcal{E,F})$ is a Dirichlet form on $L^2(\mathbf{R}^d,\rho dx)$ (Cf.  \cite{KO}), and clearly 
\[
	C_c^\infty(\mathbf{R}^d)\subset \mathcal{F}.
\]
We will always make the following hypothesis:
\begin{itemize}
\item[(\bf{H})] The Dirichlet form $(\mathcal{E,F})$ is regular on $L^2(\mathbf{R}^d,\rho dx)$, and $C_c^\infty(\mathbf{R}^d)$ is a core of $(\mathcal{E,F})$. Furthermore, there exists a positive measurable function $\hat{\rho}$ on $[0,\infty)$ such that 
\[
\rho(x)=\hat{\rho}(|x|)
\]
for any $x\in \mathbf{R}^d$. In other words, $\rho$ is a radius function on $\mathbf{R}^d$.
\end{itemize}
For example, when $\rho\equiv 1$, then $(\mathcal{E,F})$ is exactly the associated Dirichlet form $(\frac{1}{2}\mathbf{D},H^1(\mathbf{R}^d))$ of $d$-dimensional Brownian motion. Apparently (\textbf{H}) is satisfied. The energy form outlined in Example~\ref{EXA2}, i.e. $d=3,\rho=\psi_\gamma^2$, where $\psi_\gamma$ is given by \eqref{EQPHI}, also satisfies (\textbf{H}). However, (\textbf{H}) is not always right. We refer some counterexamples, as well as some sufficient conditions to ensure the denseness of $C_c^\infty(\mathbf{R})$ in $\mathcal{F}$, to \cite{KO} \cite{KA} \cite{ZV} and \cite{TBO}. 

Denote the corresponding diffusion of $(\mathcal{E,F})$ by $X$. The part process $X^0$ on $\mathbf{R}^d\setminus \{\mathbf{0}\}$ of $X$ corresponds to a regular Dirichlet form $(\mathcal{E}^0,\mathcal{F}^0)$ on $L^2(\mathbf{R}^d\setminus \{\mathbf{0}\},\rho dx)$. 
In particular, $C_c^\infty(\mathbf{R}^d\setminus \{\mathbf{0}\})$ is a core of $(\mathcal{E}^0,\mathcal{F}^0)$. Note that if $\{\mathbf{0}\}$ is $\rho dx$-polar with respect to $X$ (such as the $d$-dimensional Brownian motion), then  $$(\mathcal{E}^0,\mathcal{F}^0)=(\mathcal{E,F}).$$ Otherwise, $\mathcal{F}^0 \neq \mathcal{F}$ (such as Example~\ref{EXA2}).

\begin{lemma}\label{LM4}
	Let $(\mathcal{E}^{(p)},\mathcal{F}^{(p)})$ be the associated Dirichlet form of minimal diffusion $r^{(p)}$ on $(0,\infty)$, whose scaling function and speed measure are $p$ and $l$. Assume that
	\begin{equation}\label{EQCII}
		C_c^\infty((0,\infty))\subset \mathcal{F}^{(p)}.
	\end{equation}
Then $C_c^\infty((0,\infty))$ is a core of $(\mathcal{E}^{(p)},\mathcal{F}^{(p)})$ if and only if $p$ is absolutely continuous.
\end{lemma}

The above lemma is taken from \cite{FL2} (see Theorem~3 of \cite{FL2}), which would be very useful to prove the existence and uniqueness of regular extension. 
Denote by $q$ the inverse function of $p$, i.e. $q=p^{-1}$. Set $J:=p((0,\infty))$. Note that \eqref{EQCII} is equivalent to that $q$ is absolutely continuous and $q'\in L^2_\mathrm{loc}(J)$. 

\begin{lemma}
Assume (\textbf{H}) holds. Then $X^0$ can be written as 
\begin{equation}\label{EQXRT}
	X^0=(r^0_t,\vartheta_{A_t})_{t\geq 0},
\end{equation}
where $r^0$ is a diffusion on $(0,\infty)$, {whose scale function $p$ and speed measure $l$ are 
\[
	\begin{aligned}
		dp(x)&=\frac{1}{|S^{d-1}|\cdot\hat{\rho}(x)\cdot x^{d-1}}dx,\\
		l(dx)&=|S^{d-1}|\cdot \hat{\rho}(x)\cdot x^{d-1}dx,
	\end{aligned}
\]
the associated Dirichlet form of $r^0$ on $L^2\left((0,\infty),l\right)$ is the closure of }
\begin{equation}\label{EQDEPL}
	\begin{aligned}
		\mathcal{D}(\mathcal{E}^{p,l})&=C_c^\infty((0,\infty)),\\
		\mathcal{E}^{p,l}(u,v)&=\frac{1}{2}\int_0^\infty u'(x)v'(x)l(dx),\quad u,v\in \mathcal{D}(\mathcal{E}^{p,l}),
	\end{aligned}
\end{equation}
$\vartheta$ is the ($\sigma$-symmetric) spherical Brownian motion on $S^{d-1}$ and 
\[
	A_t=\int_0^t \frac{1}{(r^0_s)^2}ds,\quad t<\zeta^0.
\]
Here, $|S^{d-1}|$ is the surface area of  unit sphere $S^{d-1}$, $\sigma$ is the uniform surface measure on $S^{d-1}$ and $\zeta^0$ is the life time of $X^0$, i.e. the hitting time of $\{\mathbf{0}\}$ relative to $X$. 
\end{lemma}
\begin{proof}
Let $q:=p^{-1}$. From \eqref{EQRHO}, we conclude that $q$ is absolutely continuous and $q'\in L^2_\mathrm{loc}$. It follows that \eqref{EQCII} holds. 
Note that for any $u\in C_c^\infty((0,\infty))$, 
\[
	\int_0^\infty (\frac{du}{dp})^2(x)dp(x)=\int_0^\infty u'(x)^2 q'(p(x))^2dp(x)=\int_0^\infty u'(x)^2l(dx).
\]
Then it follows from Lemma~\ref{LM4} that the minimal diffusion on $(0,\infty)$ with scaling function $p$ and speed measure $l$ can be characterized by \eqref{EQDEPL}.

On the other hand, since $\mathbf{R}^d\setminus \{0\}=(0,\infty)\times S^{d-1}$, it follows from 
\[
	\Delta=\frac{1}{r^{d-1}}\frac{\partial}{\partial r}(r^{d-1}\frac{\partial}{\partial r})+\frac{1}{r^2}\Delta_{S^{d-1}}
\]
that for any $u\in C_c^\infty(\mathbf{R}^d\setminus \{0\})$, 
\[
\begin{aligned}
	\mathcal{E}^0(u,u)=& \frac{1}{2}\int_{S^{d-1}}\int_0^\infty (\frac{\partial u}{\partial r})(r,y)l(dr)\sigma(dy)\\&+\int_0^\infty \frac{1}{2}(-\Delta_{S^{d-1}} u(r,\cdot), u(r,\cdot))_\sigma \mu_A(dr),
\end{aligned}\]
where $r$ is the radius coordinate and $\Delta_{S^{d-1}}$ is the Laplace-Beltrami operator on $S^{d-1}$. From Theorem~\ref{THM1}, we obtain \eqref{EQXRT}. That completes the proof. 
\end{proof}

Our main result of this section is as follows.

\begin{theorem}\label{THM3}
Assume (\textbf{H}) holds. Let $(\tilde{\mathcal{E}}^0,\tilde{\mathcal{F}}^0)\prec(\mathcal{E}^0,\mathcal{F}^0)$ on $L^2(\mathbf{R}^d\setminus \{\mathbf{0}\},\rho dx)$, and  its associated diffusion  $\tilde{X}^0\in \mathcal{R}$. Then there exists a unique regular extension $(\tilde{\mathcal{E}},\tilde{\mathcal{F}})$ of $(\tilde{\mathcal{E}}^0,\tilde{\mathcal{F}}^0)$ on $\mathbf{R}^d$, i.e. 
\[
	(\tilde{\mathcal{E}},\tilde{\mathcal{F}})\prec (\mathcal{E,F}),
\]
and $(\tilde{\mathcal{E}}^0,\tilde{\mathcal{F}}^0)$ is exactly the part Dirichlet form of $(\tilde{\mathcal{E}},\tilde{\mathcal{F}})$ on $\mathbf{R}^d\setminus \{\mathbf{0}\}$.
\end{theorem}

\begin{proof}
Assume	$(\tilde{\mathcal{E}},\tilde{\mathcal{F}})$ and $(\mathcal{E}',\mathcal{F}')$ are both regular extensions of $(\tilde{\mathcal{E}}^0,\tilde{\mathcal{F}}^0)$ and $\tilde{X}, X'$ are their corresponding diffusions respectively. It follows from Theorem~1 of  \cite{LY} that $\tilde{X},X'$ have no killing inside. Thus $\tilde{X},X'$ are both the \emph{one-point extensions} (Cf. Definition~7.5.1 of  \cite{CM}) of $\tilde{X}^0$. On the other hand, from Corollary~\ref{COR3}, we can write
\[
	\tilde{X}^0=(\tilde{r}^0_t,\vartheta_{\tilde{A}_t})_{t\geq 0},
\]
where $\tilde{r}^0\prec r$ and $\tilde{A}_t=\int_0^t 1/(\tilde{r}^0_s)^2ds$. Let $\tilde{p}$ be the scale function of $\tilde{r}^0$. It follows from  \cite{XPJ} that $d\tilde{p}\ll dp$ and $d\tilde{p}/dp=0$ or $1$ a.e. Note that $$p(0+):=\lim_{x\downarrow 0} p(x)$$ may equal $-\infty$ (such as Example~\ref{EXA1}) or be finite (such as Example~\ref{EXA2}). If $p(0+)>-\infty$, it follows that $\tilde{p}(0+)>-\infty$. If $p(0+)=-\infty$, then $\tilde{p}(0+)$ may be finite or infinite (see Example~5.2 of  \cite{XPJ} for an example of finite case). 
Consider the first case:
	\[
		\tilde{p}(0+)=-\infty.
	\]
	Note that from \eqref{EQRHO}, we have $l((0,c))<\infty$ for any $c>0$.
Then $0$ is an unapproachable boundary point with respect to $\tilde{r}^0$. Thus $\tilde{r}^0$ cannot hit $0$ at finite time. More precisely, it implies that
\[
	P^x_{\tilde{r}^0}(\tilde{\sigma}_0<\infty)=0,\quad x\in (0,\infty),
\]
where $P^x_{\tilde{r}^0}$ is the probability measure of $\tilde{r}^0$ starting from $x$, and $\tilde{\sigma}_0$ is the hitting time of $\{0\}$ relative to $\tilde{r}^0$. Since $\tilde{r}^0=|\tilde{X}^0|$, we can conclude that
\[
	\tilde{P}_0^x(\tilde{T}_0<\infty)=0,\quad x\in \mathbf{R}^d\setminus \{\mathbf{0}\},
\]
where $\tilde{P}_0^x$ is the probability measure of $\tilde{X}^0$, and $\tilde{T}_0$ is the hitting time of $\{\mathbf{0}\}$ relative to $\tilde{X}^0$. Hence $\{\mathbf{0}\}$ is a $\rho dx$-polar set with respect to $X'$ or $\tilde{X}$ because they coincide with $\tilde{X}^0$ on $\mathbf{R}^d\setminus \{\mathbf{0}\}$. Then we have $X'=\tilde{X}$.
Now consider the second case: 
\[
	\tilde{p}(0+)>-\infty.
\]
In other words, $0$ is a regular boundary point of $\tilde{r}^0$. Thus 
\[
	P^x_{\tilde{r}^0}(\tilde{\sigma}_0<\infty,\tilde{r}^0_{\tilde{\sigma}_0-}=0)>0,\quad x\in (0,\infty).
\]
Then 
\[
	\tilde{P}_0^x(\tilde{T}_0<\infty, \tilde{X}^0_{\tilde{T}_0-}=\mathbf{0})>0,\quad x\in \mathbf{R}^d\setminus \{\mathbf{0}\},
\]
since $\tilde{X}^0$ is rotationally invariant. It follows from Theorem~7.5.4 of  \cite{CM} that the one-point extension of $\tilde{X}^0$ on $\mathbf{R}^d$ is unique, i.e. $X'=\tilde{X}$.

{Next, we turn to the proof of the existence of regular extension of 	$(\tilde{\mathcal{E}}^0,\tilde{\mathcal{F}}^0)$.}
First we still consider the case: 
\[
	\tilde{p}(0+)=-\infty.
\]
Let $M=\tilde{p}(\infty):=\lim_{r\uparrow \infty}\tilde{p}(r)$ and $\tilde{q}:=\tilde{p}^{-1}$. Since $\tilde{p}$ is strictly increasing and continuous, it follows that $\tilde{q}$ is also a strictly increasing and continuous function on $(-\infty,M)$. Denote the Lebesgue decomposition of $d\tilde{q}$ with respect to the Lebesgue measure $dr$ on $(-\infty, M)$ by 
\[
	d\tilde{q}(r)=g(r)dr+c(dr),
\]
{where $g\in L_{\text{loc}}^1((-\infty,M))$ is non-negative}, and $c$ is another Radon measure  supported on a set $H$ of zero Lebesgue measure. Note that $g$ is strictly positive almost everywhere. In fact, let $Z_g:=\{r\in (-\infty,M)\setminus H:g(r)=0\}$. Since
\[
	d\tilde{q}(Z_g)=\int_{Z_g}g(r)dr=0
\]
and $dr=\tilde{p}'(\tilde{q})d\tilde{q}$, we can deduce that the Lebesgue measure $|Z_g|$ of $Z_g$ is 0.  Define a function $\hat{q}$ on $(-\infty, M)$ by
\begin{equation}\label{EQQGT}
	\hat{q}(r):=\int_{-\infty}^r (g(t)\wedge \text{e}^t)dt,\quad r\in (-\infty,M).
\end{equation}
Clearly $\hat{q}(-\infty)=\tilde{q}(-\infty)=0$ and $\hat{q}(r)\leq \tilde{q}(r)$ for any $r\in (-\infty, M)$. Moreover, $\hat{q}$ is strictly increasing and absolutely continuous. Its inverse function $\hat{p}:=\hat{q}^{-1}$ is also strictly increasing and absolutely continuous because the Lebesgue measure of $Z_{\hat{q}'}:=\{r\in (-\infty,M):\hat{q}'(r)=0\}$ is zero. Let $L=\hat{q}(M)$ and $B_L:=\{x\in \mathbf{R}^d:|x|<L\}$. We have
\begin{equation}\label{EQPQ}
\hat{q}(\tilde{p}(r))\leq \tilde{q}(\tilde{p}(r))=r,\quad r\in (0,\infty).
\end{equation} 
Denote a measure $\hat{l}(dr):=\hat{q}'(r)^2dr$ on $(-\infty,M)$. Then for any $r\in (-\infty, M)$, we have
\begin{equation}\label{EQRET}
	\hat{l}((-\infty,r])\leq \int_{-\infty}^r \text{e}^{2t}dt<\infty. 
\end{equation}
Define a transform $h:[0,\infty)\rightarrow [0,L)$ by $h(0)=0$, $h(r)=\hat{q}(\tilde{p}(r))$ for $r>0$ and another transform $T:\mathbf{R}^d\rightarrow B_L$ by
\[
	Tx=(h(r),\theta,\varphi),
\] 
where $x=(r,\theta,\varphi)$ is the spherical coordinate of $x$. We claim that 
\begin{equation}\label{EQBLT}
	C_c^\infty(B_L)\circ T:=\{u=\psi\circ T:\psi\in C_c^\infty(B_L)\}\subset \mathcal{F}.
\end{equation}
To this end, note that 
\begin{equation}\label{EQNF}
	|\nabla f|^2(x)=(\frac{\partial f}{\partial r})^2+\frac{1}{r^2}[(\frac{\partial f}{\partial \theta})^2+\frac{1}{\sin^2\theta}(\frac{\partial f}{\partial \varphi})^2]
\end{equation}
for $x\neq 0$ and appropriate function $f$. We denote 
\[
	(\frac{\partial f}{\partial \sigma})^2:=(\frac{\partial f}{\partial \theta})^2+\frac{1}{\sin^2\theta}(\frac{\partial f}{\partial \varphi})^2
\]
for convenience. Then for any $u=\psi\circ T\in C_c^\infty(B_L)\circ T$, since $u$ is continuous with compact support, it follows that $u\in L^2(\mathbf{R}^d,\rho dx)$. On the other hand, from \eqref{EQNF} and Proposition~1 of  \cite{FL}, we can deduce that
\[
\begin{aligned}
	\int |\nabla u|^2(x)\rho(x)dx&=\int_{S^{d-1}}\int_0^\infty (\frac{\partial u}{\partial r})^2l(dr)d\sigma +\int_0^\infty \int_{S^{d-1}} (\frac{\partial u}{\partial \sigma})^2d\sigma d\mu \\
	&=\int_{S^{d-1}}\int_0^\infty (\frac{\partial u}{\partial \tilde{p}})^2 d\tilde{p}d\sigma+\int_0^L\int_{S^{d-1}}(\frac{\partial \psi}{\partial \sigma})^2d\sigma d\mu\circ h^{-1}.
\end{aligned}\]
But 
\[\begin{aligned}
	\int_0^\infty (\frac{\partial u}{\partial \tilde{p}})^2 d\tilde{p}&=\int_{-\infty}^M (\frac{\partial \psi\circ \hat{q}}{\partial r})^2 dr\\&=\int_{-\infty}^M (\frac{\partial \psi}{\partial r}(\hat{q}(r))^2\hat{q}'(r)^2dr\\&=\int_0^L (\frac{\partial \psi}{\partial r})^2d\hat{l}\circ \hat{p}.
\end{aligned}\]
Thus 
\[
	\begin{aligned}
		\int |\nabla u|^2(x)\rho(x)dx&\leq \int_{B_L} |\nabla \varphi|^2 d(\hat{l}\circ \hat{p}+r^2\cdot\mu\circ h^{-1})d\sigma.
	\end{aligned}
\]
Note that for any $t\in (0,L)$, it follows from \eqref{EQPQ} and \eqref{EQRET} that 
\[
	\hat{l}\circ \hat{p}((0,t])= \hat{l}(({-\infty},{\hat{p}}(t)])<\infty
\]
and 
\[
	r^2\mu\circ h^{-1}((0,t])=\int_0^{h(t)}\frac{h(r)^2}{r^2}l(dr)\leq l((0,h(t)])<\infty.
\]
Hence
\[
	\int |\nabla u|^2(x)\rho(x)dx\leq \infty
\]
and \eqref{EQBLT} is proved. Then $(\mathcal{E},C_c^\infty(B_L)\circ T)$ is closable on $L^2(\mathbf{R}^d,\rho dx)$ and its closure, denoted by $(\hat{\mathcal{E}},\hat{\mathcal{F}})$, is a regular subspace of $(\mathcal{E,F})$. We will illustrate that $ (\hat{\mathcal{E}},\hat{\mathcal{F}})$ is exactly the required Dirichlet form $(\tilde{\mathcal{E}},\tilde{\mathcal{F}})$. In fact, the part Dirichlet form $(\hat{\mathcal{E}}^0,\hat{\mathcal{F}}^0)$ on $\mathbf{R}^d\setminus \{\mathbf{0}\}$ of $(\hat{\mathcal{E}},\hat{\mathcal{F}})$ is regular with the core $C_c^\infty(B_L\setminus \{\mathbf{0}\})\circ T$. Let $\hat{X}^0$ be the associated Markov process of $(\hat{\mathcal{E}}^0,\hat{\mathcal{F}}^0)$. We can deduce easily from Theorem~\ref{THM1} that $T(\hat{X}^0)$ can be written as
\[
	T(\hat{X}^0)=(\hat{r}^0_t,\vartheta_{\hat{A}_t})_{t\geq 0},
\]
where $\hat{r}^0$ is a diffusion on $(0,L)$ with scale function $\hat{p}$ and speed measure $l\circ h^{-1}$, and $\hat{A}$ is the PCAF of $\hat{r}^0$ whose associated Revuz measure is $\mu\circ h^{-1}$. Then 
\[
	\hat{X}^0=(h^{-1}(\hat{r}^0_t),\vartheta_{\hat{A}_t})_{t\geq 0}:=(\tilde{r}^0_t,\vartheta_{\hat{A}_t})_{t\geq 0}
\]
where $\tilde{r}^0$ is a diffusion on $(0,\infty)$ with the scale function $\tilde{p}$ and the speed measure $l$, and $\hat{A}$ is the PCAF of $\tilde{r}^0$, whose corresponding Revuz measure is $\mu$. Therefore, we obtain $\hat{X}^0=\tilde{X}^0$. 

Now we consider the second case:
\[
	\tilde{p}(0+)>-\infty.
\]
Without loss of generality,  assume that $\tilde{p}(0+)=0$. We can replace \eqref{EQQGT} by 
\[
	\hat{q}(r):=\int_0^r (g(t)\wedge t)dt,\quad r\in (0,M)
\] 
and prove this case through the same way as the first case.    
\end{proof}

\begin{remark}
	Since $(\mathcal{E}^0,\mathcal{F}^0)$ is irreducible, it follows from Lemma~7.5.3 of  \cite{CM} and Remark~\ref{RM4} that $(\mathcal{E,F})$, $(\tilde{\mathcal{E}}^0,\tilde{\mathcal{F}}^0)$ and $(\tilde{\mathcal{E}},\tilde{\mathcal{F}})$ are all irreducible. But other global properties may be different between  regular subspace on $\mathbf{R}^d\setminus \{\mathbf{0}\}$ and its regular extension on $\mathbf{R}^d$. For instance, as outlined in Example~\ref{EXA2}, the Dirichlet form $(\mathcal{E,F})$ given by \eqref{EQFEF} is recurrent because $1\in \mathcal{F}$ and $\mathcal{E}(1,1)=0$. However, from Example~\ref{EXA4}, we know that its part Dirichlet form on $\mathbf{R}^3\setminus \{\mathbf{0}\}$ is transient.
\end{remark}

Note that the first kind of regular subspace, i.e. $\tilde{p}(0+)=-\infty$, can only happen when $p(0+)=-\infty$. That means if $\{\mathbf{0}\}$ is a $\rho dx$-polar set with respect to $\tilde{X}^0$, then  the regular subspace $(\tilde{\mathcal{E}}^0,\tilde{\mathcal{F}}^0)$ of $(\mathcal{E}^0,\mathcal{F}^0)$ on $L^2(\mathbf{R}^d\setminus \{0\},\rho dx)$ is  regular on $L^2(\mathbf{R}^d,\rho dx)$. Hence it is also a regular subspace of $(\mathcal{E,F})$ on $L^2(\mathbf{R}^d,\rho dx)$. Moreover, the second case may happen even if $p(0+)=-\infty$. Thus we have the following corollary.

\begin{corollary}\label{COR6}
	Under the same assumptions and notations in Theorem~\ref{THM3}, assume in addition that $\{\mathbf{0}\}$ is $\rho dx$-polar with respect to $X$ but not $\rho dx$-polar with respect to $\tilde{X}^0$. Then $\mathcal{F}^0=\mathcal{F}$, and there exists a unique regular extension $(\tilde{\mathcal{E}},\tilde{\mathcal{F}})$  on $\mathbf{R}^d$ of $(\tilde{\mathcal{E}}^0,\tilde{\mathcal{F}}^0)$ such that $\tilde{\mathcal{F}}^0\neq \tilde{\mathcal{F}}$.
\end{corollary}

In a word, in the context of rotationally invariant diffusions, if we want to consider the problems about the regular subspaces on $\mathbf{R}^d$, we may first consider the accordant problems on $\mathbf{R}^d\setminus \{\mathbf{0}\}$, then there always exists  a unique regular extension onto $\mathbf{R}^d$. Thus it does not matter in Example~\ref{EXA3} that we only make the discussions on regular subspaces of  Brownian motion on $\mathbf{R}^d\setminus \{\mathbf{0}\}$, because every such regular subspace, say $(\mathcal{E}',\mathcal{F}')$, uniquely corresponds to a regular subspace of $(\frac{1}{2}\mathbf{D},H^1(\mathbf{R}^d))$. If $$p'(0+)=-\infty$$ (the notation employed in Example~\ref{EXA3}), then $(\mathcal{E}',\mathcal{F}')$ itself is also regular on $L^2(\mathbf{R}^d)$. Hence  $$(\mathcal{E}',\mathcal{F}')\prec (\frac{1}{2}\mathbf{D},H^1(\mathbf{R}^d)).$$ If $p'(0+)>-\infty$, then the regular extension of $(\mathcal{E}',\mathcal{F}')$ differs to itself and can be constructed through the same way as Theorem~\ref{THM3}.

\section*{Acknowledgement}
This work was initiated when the first author visited the University of California, San Diego. He would like to thank Professor Patrick J. Fitzsimmons for his hospitality and many helpful discussions. {We also want to thank anonymous reviewers for pointing out the article \cite{OH97} that we missed before. }




\end{document}